\documentclass[reqno]{amsart}
\usepackage{a4wide,amssymb}
\usepackage{hyperref}
\usepackage{float}

\vspace{8mm}

\begin{document}

\title[ on a nonclassical boundary  problem for the Laplace equation]
{On a nonclassical boundary  problem for the Laplace equation in a
 cracked plane}

\author[M. Krasnoshchok]
{Mykola Krasnoshchok}

\address{Institute of Applied Mathematics and Mechanics of NASU
\newline\indent
G.Batyuka st. 19, 84100 Sloviansk, Ukraine} \email[M.
Krasnoshchok]{iamm012@ukr.net}


\begin{abstract}
We consider the Laplace equation in a cracked plane with a nonclassical boundary conditions. This problem arises as a model of the flow in the fractured media.  The main result is the theorem of existence and uniqueness of a solution in weighted Sobolev spaces.
\end{abstract}

\maketitle

\numberwithin{equation}{section}
\newtheorem{theorem}{Theorem}[section]
\newtheorem{lemma}[theorem]{Lemma}
\newtheorem{definition}[theorem]{Definition}
\newtheorem{proposition}[theorem]{Proposition}
\newtheorem{remark}[theorem]{Remark}
\allowdisplaybreaks

\section{Introduction}
\label{s1}

 \noindent Consider the plane containing a crack $\Sigma=\left\{(x_1,x_2): x_1<0, x_2=0\right\}$. Let us denote  by $\gamma_+$ and $\gamma_-$ the trace operators on $\Sigma$ from the upper and lower sides of $\Sigma$ respectively. This paper is concerned with the problem
 to find a pair $(\mathrm{p}(x),\mathrm{q}(x_1))$, such that
\begin{equation}\label{g1.1}
\Delta\mathrm{p}(x)=0, \mbox{ in } \Omega=\mathbb{R}^2\setminus \Sigma,
\end{equation}
\begin{equation}\label{g1.2}
\gamma_+\mathrm{p}+\gamma_-\mathrm{p}-2\mathrm{q}=0, \mbox{ on  } \Sigma,
\end{equation}
\begin{equation}\label{g1.3}
\gamma_+(\kappa_1  \mathrm{p}_{x_2}-\mathrm{p})+\gamma_-(\kappa_1
\mathrm{p}_{x_2}+ \mathrm{p})=0, \mbox{ on  } \Sigma,
\end{equation}
\begin{equation}\label{g1.4}
-\mathrm{q}_{x_1x_1}-\kappa_2(\gamma_+\mathrm{p}_{x_2}-
\gamma_- \mathrm{p}_{x_2})=\mathrm{f}, \mbox{ on  } \Sigma,
\end{equation}
\begin{equation}\label{g1.5}
\mathrm{q}_{x_1}(0)=0
\end{equation}
\begin{equation}\label{g1.6}
\mathrm{p}, \mathrm{q}\to 0, \quad |x|\to \infty,
\end{equation}
here $\kappa_1>0$ and $\kappa_2>0$.

We may consider \eqref{g1.1}-\eqref{g1.6} as a model problem for the
flow in the fractured media (cf. system (4.1) in \cite{MJR} with
parameter $\xi=1/2$). In this context, $\mathrm{q}$ is the pressure  of a fluid in the fracture $\Sigma$ and $\mathrm{p}$ is a pressure in the surrounding media $\Omega$ (cf. \cite{ABH}).

Different aspects of the theory of linear elliptic boundary value problems in non-smooth domains were discussed in \cite{KO1}, \cite{BK1},
\cite{GR}, \cite{KS2}, \cite{NZ}, and, in a more special case of cracks, in
\cite{DG1},\cite{CD}, \cite{KS1}, \cite{KHVK}, \cite{KHKZ}, \cite{ACN},
\cite{MN1} and in the references therein.  The papers
\cite{M1},\cite{M2}, \cite{KT1} are devoted to the application of
the potential theory to the third boundary problem for the Laplace equation on a bounded planar domain with inside crack.

In \cite{SF1}, the authors undertook a detailed study of the Poisson's equation in   infinite planar angle under the Dirichlet boundary condition on the one side of the angle and the oblique derivative boundary condition with a complex parameter on the other side of the angle. Using the Mellin transform, they reduced that problem to a first order difference equation in the complex plane. Their approach provides a unified way of treating a whole variety of elliptic problems with nonhomogeneous symbol in corner domains (see \cite{SF2}, \cite{BF}, \cite{BV1}, \cite{DV}, \cite{JBJ} and the references therein). Our argument is an adaptation of some ideas in \cite{SF1}, \cite{BF}. We repeat the material from \cite{SF1} without proof, thus making our exposition self-contained. We emphasize that the main difficulty consists in verifying of boundary condition \eqref{g1.5}.

The main result of this paper is the following theorem (see Section
2 for the definitions of the weighted Sobolev spaces $H^k_\mu(\Omega)$,
$H^{k+1/2}_\mu(\Sigma)$).

\begin{theorem}\label{BT2.1}
Let $k\geq 0$, $\mu\in \mathbb{R}$, and
\begin{equation}\label{g2.1}
0<\mu-k-1<\frac 1 2.
\end{equation}
Let $\mathrm{f}\in H^{k+1/2}_\mu(\Sigma)$ and, additionally,
\begin{equation}\label{g2.2}
\int\limits_{-\infty}^0\mathrm{f}(x_1)dx_1=0.
\end{equation}
Then there exists a unique solution
$(\mathrm{p}(x),\mathrm{q}(x_1))$ to the problem
\eqref{g1.1}-\eqref{g1.6} such that
\begin{equation}\label{g2.3}
\Vert \mathrm{p}\Vert_{H^{k+2}_\mu(\Omega)}+\Vert
\mathrm{q}\Vert_{H^{k+3/2}_\mu(\Sigma)}+\Vert
\mathrm{q}_x\Vert_{H^{k+3/2}_\mu(\Sigma)}\leq C \Vert
\mathrm{f}\Vert_{H^{k+1/2}_\mu(\Sigma)},
\end{equation}
here $C$ does not depend on $f$.
\end{theorem}

The paper is organized as follows. In Section 2, we define functional spaces and give some auxiliary results. The proof of the main result is splitted in two parts. Section 3 is devoted to integral representation of solution and estimate  \eqref{g2.3}. In  Section 4, we verify that obtained solution satisfies boundary condition \eqref{g1.5}. Appendices contain the proof of some technical results. Boundary conditions \eqref{g1.2}-\eqref{g1.4} are verified in Appendix A. Appendix B gives the proof of a key estimate, which we use in Section 4.

Throughout below, the symbol $C$ will denote a generic positive constant, depending only on the structural quantities of the problem.

\section{Functional spaces. Auxiliary results}
\label{s2}

Let us define the functional spaces used below (cf. \cite{GSV}, \cite{KO1}, \cite{NZ}, \cite{KMR2}). For $k\in \mathbb{N}$,  we define the Hilbert space $H^k(-\pi,\pi)$ with the norm
\[
\Vert u \Vert_{H^k(-\pi,\pi)}=\Biggl(\sum\limits^k_{l=0}\int\limits^\pi_{-\pi}\biggl|\frac{d^lu(\theta)}{d\, \theta^l}\biggr|^2d\theta\Biggr)^{\frac 1 2}.
\]Define the space
$H^k_\mu(\Omega)$ as  the closure of infinitely differentiable functions defined  in $\Omega$, vanishing near the origin and for large $|x|$, with respect to the norm  $(k\in \mathbb{N}\cup 0, \mu\in \mathbb{R})$
\begin{equation}\label{g3.1}
\Vert
\mathrm{p}\Vert_{H^k_\mu(\Omega)}=\left(\sum\limits_{|\alpha|\leq
k}\int\limits_\Omega
|x|^{2(\mu-k+|\alpha|)}|D^\alpha_x\mathrm{p}(x)|^2dx \right)^{\frac
1 2}.
\end{equation}

Passing to polar coordinates ($r>0$, $\theta\in (-\pi,\pi]$)
\begin{equation}\label{g3.2}
 x_1=r \cos \theta, \quad x_2=r \sin \theta, \quad \mathrm{p}(x)=p(r,\theta),
\end{equation}
 we define the norm
\begin{equation}\label{g3.3}
\Vert p \Vert_{k,\mu}=\left(\int\limits^\infty_0
r^{2(\mu-k)+1}\sum\limits^k_{l=0}\Vert \left(
r\frac{\partial}{\partial r}\right)^l p\Vert
^2_{H^{k-l}(-\pi,\pi)}dr\right)^{\frac 1 2},
\end{equation}
which is equivalent to the norm \eqref{g3.1} (see formula (6.1.4) in
\cite{KMR2}).

Define also $H^{k+1/2}_\mu(\mathbb{R_+})$ as the closure of $C^\infty$ functions with compact support in the norm
\begin{equation}\label{g3.4}
\Vert
h\Vert_{H^{k+1/2}_\mu(\mathbb{R_+})}=\left(\sum\limits^k_{l=0}\int\limits^\infty_0
r^{2(\mu-k+l)-1}\left|\frac{d^lh(r)}{dr^l}\right|^2dr+ \Vert
h\Vert_{L^{k+1/2}_\mu(\mathbb{R_+})}\right)^{\frac 1 2},
\end{equation}
where
\begin{equation}\label{g3.5}
\Vert
h\Vert_{L^{k+1/2}_\mu(\mathbb{R_+})}=\left(\int\limits^\infty_0
r^{2\mu}dr\int\limits^r_0\left|\frac{d^kh(r+\rho)}{dr^k}-\frac{d^kh(r)}{dr^k}\right|^2\frac{d\rho}{\rho^2}\right)^{\frac
1 2}.
\end{equation}


Let us say that a function $\mathrm{h}(x_1)$, being defined on $\Sigma$, belongs to $H^{k+1/2}_\mu(\Sigma)$ if $h(r)=\mathrm{h}(-r)$
belongs to $H^{k+1/2}_\mu(\mathbb{R}_+)$.

Define the Mellin transformation and its inverse
\begin{align}\label{g3.7}
\widetilde{h}(s)=\int\limits^\infty_0r^{s-1}h(r)dr,  s=s_1+is_2,\nonumber\\
h(r)=\frac 1 {2\pi
i}\int\limits^{s_1+i\infty}_{s_1-i\infty}r^{-s}\widetilde{h}(s)ds,
\end{align}
(see \cite{T1}, \cite{SF1} for theoretical background).

We will use the Parceval formula
\begin{equation}\label{g4.1}
\int\limits^\infty_0|h(r)|^2r^{2a-1}dr=\frac 1 {2\pi
i}\int\limits^{a+i\infty}_{a-i\infty}|\widetilde{h}(s)|^2\frac{ds}{i}
\end{equation}
and the Mellin convolution formula
\begin{equation}\label{g4.2}
\int^\infty_0h_1\left(\frac r t\right)h_2(t)\frac{dt}t=\frac 1 {2\pi
i}\int\limits^{c+i\infty}_{c-i\infty}r^{-s}\widetilde{h_1}(s)\widetilde{h_2}(s)ds,
\end{equation}
(see (3.1.15) and (3.1.13) in \cite{PK1}).

Upon integrations by parts one can prove that for any $h\in C^\infty_0(\mathbb{R}_+)$
\begin{equation}\label{g4.3}
\left(\widetilde{\frac {d^kh}{dr^k}}\right)(s)=(-1)^k(s-1)\dots(s-k)\widetilde{h}(s-k),
\end{equation}
\begin{equation}\label{g4.4}
\left(\widetilde{(r\frac \partial {\partial
r})^kh}\right)(s)=s^k\widetilde{h}(s).
\end{equation}

For $h\in H^{k+1/2}_\mu(\mathbb{R}_+)$, we have (see Proposition 3.2 in
\cite{SF1})
\begin{equation}\label{g4.5}
c_1(\mu)\int\limits^{\mu-k+i\infty}_{\mu-k-i\infty}|\widetilde{h}(s)|^2(1+|s|)^{2k+1}\frac{ds}i\leq
\Vert h\Vert^2_{H^{k+1/2}_\mu(\mathbb{R}_+)}\leq
c_2(\mu)\int\limits^{\mu-k+i\infty}_{\mu-k-i\infty}|\widetilde{h}(s)|^2(1+|s|)^{2k+1}\frac{ds}i,
\end{equation}
here $0<c_1\leq c_2$.

\section{Integral representation of solution and its estimate}

Let us set
\begin{equation}\label{g3.6}
    q(r)=\mathrm{q}(-r), \quad f(r)=\mathrm{f}(-r).
\end{equation}
Next, we rewrite \eqref{g1.1}-\eqref{g1.6} by adopting polar
coordinates \eqref{g3.2} and  \eqref{g3.6}
\begin{equation}\label{g5.1}
p_{rr}+\frac 1 r p_r+\frac 1 {r^2}p_{\theta\theta}=0, r>0, \theta\in (-\pi,\pi)
\end{equation}
\begin{equation}\label{g5.2}
p(r,\pi)+p(r,-\pi)-2q(r)=0, r>0,
\end{equation}
\begin{equation}\label{g5.3}
\frac {\kappa_1} {r} p_\theta(r,\pi)+p(r,\pi)+\frac {\kappa_1} {r}
p_\theta(r,-\pi)-p(r,-\pi)=0, r>0,
\end{equation}
\begin{equation}\label{g5.4}
-q_{rr}(r)+\frac{\kappa_2}r(p_{\theta}(r,\pi)- p_{\theta}(r,-\pi))=f(r), r>0,
\end{equation}
\begin{equation}\label{g5.5}
q_{r}(0)=0
\end{equation}
\begin{equation}\label{g5.6}
p, q\to 0, \quad r\to \infty.
\end{equation}

The Mellin transformation with respect to $r$ and formulas \eqref{g4.3}, \eqref{g4.4} lead
\eqref{g5.1}-\eqref{g5.4}  to the system
\begin{equation}\label{g5.7}
\widetilde{p}(s,\theta)+s^2\widetilde{p}(s,\theta)=0,
\end{equation}
\begin{equation}\label{g5.8}
\widetilde{p}(s,\pi)+\widetilde{p}(s,-\pi)-2\widetilde{q}=0,
\end{equation}
\begin{equation}\label{g5.9}
\kappa_1\left[  \widetilde{p}_\theta(s-1,\pi)+
\widetilde{p}_\theta(s-1,-\pi)\right]+\widetilde{p}(s,\pi)-\widetilde{p}(s,-\pi)=\widetilde{q}(s),
 \end{equation}
\begin{equation}\label{g5.10}
-(s-1)(s-2)\widetilde{q}(s-2)+\kappa_2(\widetilde{p}_{\theta}(s-1,\pi)-\widetilde{p}_{\theta}(s-1,-\pi))=\widetilde{f}(s)
.
\end{equation}

It seems appropriate to seek for $\widetilde{p}$, $\widetilde{q}$ in the form
\begin{equation}\label{g5.11}
\widetilde{p}(s,\theta)=\frac {\cos s\theta}{\sin
s\pi}d(s), \quad \widetilde{q}(s)=\cot s\pi d(s),
\end{equation}
with unknown $d(s)$. We have
\begin{align}\label{g5.12}
\widetilde{p}(s,\pm\pi)&=\cot s\pi d(s),
\nonumber \\
\widetilde{p}_\theta(s,&\theta)=-s\frac{\sin s\theta}{\sin s\pi}d(s),
\\ \widetilde{p}_\theta(s,&\pm\pi)=\mp sd(s).\nonumber
\end{align}
This allows us to satisfy \eqref{g5.7}-\eqref{g5.9} simultaneously and convert \eqref{g5.10} to the following difference equation
\begin{equation}\label{g5.13}
-(s-1)(s-2)\cot (s-2)\pi d(s-2)-2\kappa_2(s-1)d(s-1)=\widetilde{f}(s).
\end{equation}
Set
\begin{equation}\label{g6.0}
\kappa_0=2\kappa_2, s=\lambda+2,
\omega(\lambda)=\lambda\cot\lambda\pi.
\end{equation}
Then equation \eqref{g5.13} takes the form
\begin{equation}\label{g6.1}
\kappa_0d(\lambda+1)+\omega(\lambda)d(\lambda)=-\frac{\widetilde{f}(\lambda+2)}{\lambda+1},
\end{equation}
These equation was extensively
studied in \cite{SF1} in a more general situation. Let us sum up some key results.

One of the solutions to homogeneous equation
\begin{equation}\label{g6.2}
\kappa_0d(\lambda+1)+\omega(\lambda)d(\lambda)=0,
\end{equation}
is
\begin{equation}\label{g6.3}
d_0(\lambda)=\exp(i\lambda\pi)(\kappa_0\pi)^{1/2-\lambda}K(\lambda),
\end{equation}
here
\begin{equation}\label{g6.4}
K(\lambda)=\prod\limits^\infty_{n=1}\frac{\Gamma(n-\frac 1
2+\lambda)\Gamma(n+1-\lambda)}{\Gamma(n+\frac 1
2-\lambda)\Gamma(n+\lambda)}\left(\frac n
{n-1/2}\right)^{2\lambda-1},
\end{equation}
and $\Gamma(\cdot)$ is the Euler gamma function.  $\Gamma(\lambda)$
has the simple poles in $\lambda=-m$, $m\in \mathbb{N}\cup 0$, and,
for example, $\Gamma(\lambda)\Gamma(\lambda+1)$ have the simple pole
in $\lambda=0$ and the pole of the second order in $\lambda=-1$.
Therefore, $K(\lambda)$ has the poles of the order $m+1$ in
$\lambda=-\frac 1 2 -m$, $\lambda=2+m$ and the zeroes of the order
$m+1$ in $\lambda=\frac 3 2+m$, $\lambda=-1-m$, $m\in \mathbb{N}\cup
0$.   That's why $K(\lambda)$  is analytic in the strip
\begin{equation}\label{g6.7}
-\frac 1 2<\mathrm{Re}\lambda< 2.
\end{equation}
\noindent$\frac 1{K(\lambda)}$ is analytic in the strip
\begin{equation}\label{g6.8}
-1<\mathrm{Re}\lambda<\frac 3 2.
\end{equation}
\noindent$K_0(\lambda)=\frac 1 {(\lambda+1)K(\lambda+1)}$ is
analytic in the strip
\begin{equation}\label{g6.9}
- 1 <\mathrm{Re}\lambda<\frac 1 2.
\end{equation}

The Corollary to Proposition 5.1 in \cite{SF1}, p 222, states
the convergence of the infinite product $K(\lambda)$ except from the poles of
$\Gamma(n-\frac 1 2 +\lambda)$, $\Gamma(n+1-\lambda)$, $n\in \mathbb{N}$.
Proposition 5.2 in \cite{SF1} gives the following asymptotic formula
\begin{equation}\label{g6.10}
K(\lambda)=|\omega(\lambda)|^{\lambda_1-1/2}\exp(-\lambda_2\arg
\omega(\lambda)+r_1(\lambda)), \lambda=\lambda_1+i\lambda_2,
\end{equation}
here $r_1(\lambda)=O(1)$, in the strip
\[
-\frac 1 2 < \mathrm{Re }<2.
\]
By definition \eqref{g6.0}, we have (cf. (6.7) in \cite{SF1})
\begin{align}\label{g6.11}
\omega(\lambda)&=\lambda_1+i\lambda_2)\left\{\frac {\sin 2\lambda_1\pi}{\cosh 2\lambda_2\pi -\cos 2\lambda_1\pi}-i \frac {\sinh 2\lambda_2\pi}{\cosh 2\lambda_2\pi -\cos 2\lambda_1\pi}\right\},\nonumber\\
&\mathrm{Re \, } \omega(\lambda)=\frac {\lambda_1\sin 2\lambda_1\pi+\lambda_2\sinh 2\lambda_2\pi}{\cosh 2\lambda_2\pi -\cos 2\lambda_1\pi},\\
&\mathrm{Im \, } \omega(\lambda)=\frac {\lambda_2\sin 2\lambda_1\pi-\lambda_1\sinh 2\lambda_2\pi}{\cosh 2\lambda_2\pi -\cos 2\lambda_1\pi}\nonumber
\end{align}
Therefore,  we obtain by consecutive steps
\begin{equation}\label{g6.14}
\mathrm{Re }\, \omega( \lambda )>0, |\mathrm{arg \,} \omega(\lambda)|<\frac \pi 2
\end{equation}
\begin{equation}\label{g6.15}\arg \omega(\lambda)=\arctan \frac {\lambda_2\sin
2\pi\lambda_1-\lambda_1\sinh 2\pi\lambda_2}{\lambda_1\sin
2\pi\lambda_1+\lambda_2\sinh 2\pi\lambda_2},
\end{equation}
\begin{equation}\label{g6.17}
|\arg \omega(\lambda)|\leq \frac C
{|\lambda_2|}.
\end{equation}
for sufficiently large $|\lambda_2|$
(see \S 6, p.228, \S7, p.231-232 in \cite{SF1}).

\begin{remark}\label{r3.1}
 We emphasize that, as opposed  to \cite{SF1}, there is a simple pole   in $\lambda=-1$ in the right-hand side of \eqref{g6.1}.
\end{remark}

As in \cite{SF1}, the solution of \eqref{g6.1} is sought in the form
\begin{equation}\label{g6.18}
d_1(\lambda)=d_0(\lambda)y(\lambda).
\end{equation}
Substitution of \eqref{g6.18} in \eqref{g6.1} gives
\[\kappa_0d_0(\lambda+1)y(\lambda+1)+\omega(\lambda)d_0(\lambda)y(\lambda)=-\frac{\widetilde{f}(\lambda+2)}{\lambda+1},\]
\[\kappa_0d_0(\lambda+1)y(\lambda+1)-\kappa_0d_0(\lambda+1)y(\lambda)=-\frac{\widetilde{f}(\lambda+2)}{\lambda+1},\]
and, to this end,
\begin{equation}\label{g7.1}
y(\lambda+1)-y(\lambda)=h(\lambda),\quad h(\lambda)=-
\frac{\widetilde{f}(\lambda+2)}{\kappa_0(\lambda+1)d_0(\lambda+1)}.
\end{equation}
Let $\mathcal{L}_{l,\varepsilon}$, $l=-1,0$ be the contour in the complex
plane consisting from the three parts
\[ \mathcal{L}_{l,\varepsilon}=\{ {\mathrm Re}\,
\lambda=l, \varepsilon<{\mathrm Im}\, \lambda<\infty\}\cup\{
{\mathrm Re}\,\lambda>l, |\lambda-l|=\varepsilon\}\cup\{ {\mathrm
Re}\, \lambda=l, -\infty<{\mathrm Im}\, \lambda<-\varepsilon\}.\]

\begin{remark}\label{r3.2} The function $h(\lambda)$ is analytic the strip \eqref{g6.9} instead of $-2<\mathrm{Re }\lambda<\frac 1 2$ in \cite{SF1}.
\end{remark}
Taking into account this observation, we can apply the Corollary to Proposition 6.1 in \cite{SF1} to conclude that the function
\begin{equation}\label{g7.2}
y(\lambda)=\frac 1
{2i}\int\limits_{\mathcal{L}_{-1,\varepsilon}}h(\lambda+\zeta)[\cot
\zeta\pi+i]d\zeta,
\end{equation}
is a solution of equation \eqref{g7.1} for all $\lambda$ such that
\begin{equation}\label{g7.21}
 {\mathrm
Re}\,\lambda\in (0,1/2),
\end{equation}
instead of $ {\mathrm
Re}\,\lambda\in (-1/2,1/2)$ in \cite{SF1} (see Remark 3.2). Here we use the same argument as in Statement 6.1 in \cite{SF1} as follows. We consider difference $y(\lambda +1)-y(\lambda)$, change the variable in the integral for $y(\lambda +1)$ and verify that $h(\lambda+\zeta)$ is analytic function with respect to $\zeta$ in a domain bounded by the contours $\mathcal{L}_{-1,\varepsilon}$ and $\mathcal{L}_{0,\varepsilon}$. Finally, we use the Cauchy residue theorem.

By \eqref{g6.18}, \eqref{g7.2}, we find a solution of \eqref{g6.1}
\begin{align}\label{g7.3}
d_1(\lambda)=d_0(\lambda)y(\lambda)=-\frac 1
{2i}\int\limits_{\mathcal{L}_{-1,\varepsilon}}\frac {d_0(\lambda)}{\kappa_0(\lambda+\zeta+1)d_0(\lambda+\zeta+1)}\widetilde{f}(\lambda+\zeta+2)[\cot
\zeta\pi+i]d\zeta.
\end{align}
One can see that $y(\lambda)$ and therefore $d_1(\lambda)$ are analytic in the strip
\begin{equation}\label{g7.4}
0<\mathrm{Re }
\lambda <\frac 3 2,
\end{equation}
(cf. \eqref{g6.7}).

 Estimate (7.3) in \cite{SF1} gives (see also \eqref{g4.5})
\begin{align}\label{g8.1}
    \int\limits_{\mathrm{Re} \lambda=\mu-k-1}&|d_1(\lambda)|^2(1+|\lambda|)^{2(k+1)+1}\frac{d\lambda}{i}\nonumber \\ &\leq C
   \int\limits_{\mathrm{Re} \lambda=\mu -k-1}\left|\frac{f(\lambda+1)}{\lambda}\right|^2(1+|\lambda|)^{2(k+1)+1}\frac{d\lambda}{i}\nonumber
   \\&
   \leq C \int\limits_{\mathrm{Re} \lambda=\mu-k-1}\left|f(\lambda+1)\right|^2(1+|\lambda|)^{2k+1}\frac{d\lambda}{i}\nonumber \\ &\leq C \int\limits_{\mathrm{Re} \lambda=\mu-k}\left|f(\lambda)\right|^2(1+|\lambda|)^{2k+1}\frac{d\lambda}{i}\leq C \Vert f \Vert_{H^{k+1/2}_\mu(\mathbb{R}_+)}.
   \end{align}
Let us remind that the right-hand term in \eqref{g6.11} is $\frac{\widetilde{f}(\lambda+2)}{(\lambda+1)}$ instead of $\widetilde{g}(\lambda+1)$ in corresponding equation (4.7) in \cite{SF1}.

General solution of \eqref{g7.1} have the form
\[d_{2}(\lambda)=d_1(\lambda)+D_0(\lambda),\]
where $D_0(\lambda)$ is a general solution of homogeneous equation \eqref{g6.2}.
We check at once that
\[
\frac{D_0(\lambda+1)}{d_0(\lambda+1)}=\frac{\omega(\lambda)D_0(\lambda)}{\omega(\lambda)d_0(\lambda)}=\frac{D_0(\lambda)}{d_0(\lambda)},
\]
i.e. $E(\lambda)=\frac{D_0(\lambda)}{d_0(\lambda)}$ is a periodic function
with period one. One can show that an entire periodic function with period 1 is either constant or grows no slower as
$\exp(2\pi |\mathrm{Im }(\lambda)|)$ as $\lambda_2\to \pm \infty$
(see Theorem 4.10 in \cite{M3}). Thus the integral
\[    \int\limits_{\mathrm{Re} \lambda=l_0+1}|D_0(\lambda)|^2(1+|\lambda|)^{2(k+1)+1}\frac{d\lambda}{i}\]
is bounded only if $E(\lambda)\equiv 0$ (cf. \eqref{g8.1}). In the same way as in \cite{SF1}, p.228-229, this proves that
$d_1$ is a unique solution of \eqref{g6.1} (see also  Lemma 3.1.7 in \cite{JBJ}).

Next, we return to the functions \eqref{g5.11}. We set
\begin{equation}\label{g8.2}
\widetilde{p}(s,\theta)=\frac {\cos s\theta}{\sin
s\pi}d_1(s), \quad \widetilde{q}(s)=\cot s\pi d_1(s).
\end{equation}
Clearly, these functions are analytic in the strip $0<\mathrm{Re}s<\frac{1}{2}$.

Let us turn to the estimate \eqref{g2.3}. First, we will estimate norms of
function $q$. Let us remark that
\begin{equation}\label{gv.0}
|\cot \pi \lambda|\leq C, \mbox{if } \mathrm{Re }\lambda=\mu-k-1.
\end{equation}
By \eqref{g4.5}, \eqref{g8.1}, we have
\begin{equation}\label{gv.1}
    \Vert q \Vert_{H^{k+\frac 3 2}_\mu(\mathbb{R}_+)}\leq C\int\limits^{\mu-k-1+i\infty}_{\mu-k-1-i\infty}|\widetilde{q}(\lambda)|^2(1+|\lambda|)^{2(k+1)+1}\frac{d\lambda}{i}\leq C \Vert f \Vert_{H^{k+\frac 1 2}_\mu(\mathbb{R}_+)}.
\end{equation}
From equation \eqref{g5.4} one can expect at least that $q_{xx}\in H^{k+\frac 1 2}_\mu(\mathbb{R}_+)$. By \eqref{ga.2} (see Appendix A), we have
\[
\frac{1}{r}(p_\theta(r,-\pi)-p_\theta(r,\pi))=\frac{1}{\pi i}\int\limits_{\mathrm{Re} \lambda=\mu-k-1}r^{-\lambda-1}\lambda d_1(\lambda)d\lambda=\frac{1}{\pi i}\int\limits_{\mathrm{Re} \lambda=\mu-k}r^{-\lambda}(\lambda-1) d_1(\lambda-1)d\lambda.
\]
Therefore, the Mellin transform of $\frac{1}{r}(p_\theta(r,-\pi)-p_\theta(r,\pi))$  is $2(\lambda-1)\lambda d_1(\lambda-1).$ By \eqref{g8.1},  \eqref{g4.5}, we obtain
\begin{align*}
\int\limits^{\mu-k+i\infty}_{\mu-k-i\infty}&|(\lambda-1) d_1(\lambda-1)|^2(1+|\lambda|)^{2k+1}\frac{d\lambda}{i}\\ 
&\leq C \int\limits^{\mu-k-1+i\infty}_{\mu-k-1-i\infty}|d_1(\lambda)|^2(1+|\lambda|)^{2(k+1)+1}\frac{d\lambda}{i}\leq C \Vert f \Vert_{H^{k+\frac 1 2}_\mu(\mathbb{R}_+)},
\end{align*}
and hence
\begin{equation}\label{gv.2}
\Vert \frac{1}{r}(p_\theta(r,-\pi)-p_\theta(r,\pi)) \Vert_{H^{k+\frac 1 2}_\mu(\mathbb{R}_+)}\leq  C \Vert f \Vert_{H^{k+\frac 1 2}_\mu(\mathbb{R}_+)}.
\end{equation}
By \eqref{g5.4}, \eqref{gv.2}, we deduce that $q_{rr}\in H^{k+\frac 1 2}_\mu(\mathbb{R}_+)$ and
\begin{equation}
    \label{gv.3}
    \Vert q_{rr} \Vert_{H^{k+\frac 1 2}_\mu(\mathbb{R}_+)}\leq C\Vert f \Vert_{H^{k+\frac 1 2}_\mu(\mathbb{R}_+)}.
\end{equation}
We will prove that $q_{r}\in H^{k+\frac 3 2}_\mu(\mathbb{R}_+)$, as it asserted in \eqref{g2.3}. In fact,  definition \eqref{g3.4} and estimate \eqref{g4.5} give
\begin{align}\label{gv.4}
\Vert
&q_r\Vert^2_{H^{k+3/2}_\mu(\mathbb{R_+})}=\sum\limits^{k+1}_{l=0}\int\limits^\infty_0
r^{2(\mu-k-1+l)-1}\left|\frac{d^{l+1}q(r)}{dr^{l+1}}\right|^2dr+ \Vert
\frac{d^{k+2}q}{dr^{k+2}}\Vert_{L^{1/2}_\mu(\mathbb{R_+})} \nonumber \\
&\qquad =\int\limits^\infty_0r^{2(\mu-k-1)-1}|q_r(r)|^2dr+\sum\limits^{k}_{j=0}\int\limits^\infty_0
r^{2(\mu-k-1+j+1)-1}\left|\frac{d^{j+2}q(r)}{dr^{j+2}}\right|^2dr+ \Vert
\frac{d^{k+2}q}{dr^{k+2}}\Vert_{L^{1/2}_\mu(\mathbb{R_+})}\nonumber \\
&\qquad =\int\limits^\infty_0r^{2(\mu-k-1)-1}|q_r(r)|^2dr+\Vert
q_{rr}\Vert^2_{H^{k+1/2}_\mu(\mathbb{R_+})}.
\end{align}
The Hardy inequality (see (0.32) in \cite{KS2}) gives
\begin{equation}\label{gv.5}
    \int\limits^\infty_0r^{2(\mu-k)-1-2}|q_r(r)|^2dr\leq C
    \int\limits^\infty_0r^{2(\mu-k)-1}|q_{rr}(r)|^2dr\leq C\Vert q_{rr}\Vert^2_{H^{k+1/2}_\mu(\mathbb{R}_+)},
\end{equation}
for $2(\mu-k)-1>1$, i.e. $\mu-k-1>0$ (see \eqref{g2.1}).
\eqref{gv.4}, \eqref{gv.5} lead to
\begin{equation}\label{gv.6}
\Vert q_r\Vert^2_{H^{k+3/2}_\mu(\mathbb{R_+})}\leq C\Vert
f\Vert^2_{H^{k+1/2}_\mu(\mathbb{R_+})}.
\end{equation}

Second, we will consider  the norm \eqref{g3.3}. Remark that $2(\mu-k-2)+1=2(\mu-k-1)-1$.  The Parceval formula \eqref{g4.1} and \eqref{g4.4} give
\begin{align}\label{g9.1}
& \Vert p \Vert^2_{k+2,\mu}=\sum\limits^{k+2}_{l=0}
\int\limits^\infty_0r^{2(\mu-k-2)+1}\int\limits^\pi_{-\pi}\left|
\left(r\frac{\partial}{\partial
r}\right)^{(l)}\partial^{(k+2-l)}_\theta p(r,\theta)\right|^2d\theta dr \nonumber
\\
& =
\sum\limits^{k+2}_{l=0}
\int\limits_{\mathrm{Re }\, s=\mu-k-1}\frac{ds}{i}\int\limits^\pi_{-\pi}\left|
s^l\, \partial^{(k+2-l)}_\theta \widetilde{p}(s,\theta)\right|^2d\theta.
\end{align}
 Formula \eqref{g8.2} and estimate \eqref{g8.1} give
\begin{align*}
& \Vert p \Vert^2_{k+2,\mu}\leq C\int\limits_{\mathrm{Re }\, s=\mu-k-1}|s|^{2l}|d_1(s)|^2 \frac{ds}{i} \int\limits^\pi_{-\pi}|s|^{2(k+2-l)}\frac{|\cos s\theta|^2+|\sin \theta s|^2}{|\sin s\pi|^2}d\theta\\
&\leq C\int\limits_{\mathrm{Re }\, s=\mu-k-1}|s|^{2(k+1)}|d_1(s)|^2 \frac{ds}{i}\int\limits^{\pi}_{-\pi}|s|^2\frac{|\cos s \theta|^2+|\sin s
\theta|^2}{|\sin s \pi|^2}d\theta.
\end{align*}

By estimate (see p.230 in \cite{SF1})
\[\int\limits^{\pi}_{-\pi}|s|^2\frac{|\cos s \theta|^2+|\sin s
\theta|^2}{|\sin s \pi|^2}d\theta\leq C(1+|s|),\]
and \eqref{g8.1}, \eqref{g4.5}, we get
\begin{align}\label{g9.4}
& \Vert p \Vert^2_{k+2,\mu}\leq C\int\limits_{\mathrm{Re }\, s=\mu-k-1}(1+|s|)^{2(k+1)+1}|d_1(s)|^2 \frac{ds}{i}\leq C\Vert f\Vert^2_{H^{k+\frac 1 2}_\mu(\mathbb{R}_+)}.
\end{align}

Summing up \eqref{gv.1}, \eqref{gv.6}, \eqref{g9.4}, we deduce estimate \eqref{g2.3}.

\section{Boundary conditions}

Boundary conditions \eqref{g1.2}-\eqref{g1.4} are considered in Appendix A. This section is devoted to \eqref{g1.5}.

In brief outline, here are the main steps of the proof. First, we
split $q$ into two terms $q_1$ and $q_2$ (see \eqref{g11.2}, \eqref{g12.4} below). Second,
we prove that $q_{1,r}(0)=0$ under assumption \eqref{g2.2}. Third,
the term $q_{2,r}$ (see \eqref{g14.1} below) is estimated by the Cauchy
inequality (see \eqref{g14.5} below). Here we use that $f\in
H^{k+1/2}_\mu(\mathbb{R}_+)$, and hence
\[
\int\limits^\infty_0|f(\rho)|^2\rho^{2(\mu-k)-1}d\rho\leq \Vert
f\Vert^2_{H^{k+1/2}_\mu(\mathbb{R}_+)}.\] To this end, we use the
appropriate changes of the contour of integration with respect to
$s$ in order to estimate integrals with the kernel $\Phi(s,\zeta)$. The kernel $\Phi(s,\zeta)$ itself is estimated in Appendix B.

We proceed by converting \eqref{g7.2} to a more appropriate form. The Cauchy residue theorem gives
\begin{align}\label{g11.1}
y(\lambda)=-\frac 1
{2i}\int\limits_{\mathcal{L}_{0,\varepsilon}}\frac{\widetilde{f}(\kappa_0\lambda+\zeta+2)}{(\lambda+\zeta+1)d_0(\lambda+\zeta+1)}[\cot
\zeta\pi+i]d\zeta+\pi \mathrm{Res}_{\zeta=0}\frac{\widetilde{f}(\kappa_0\lambda+\zeta+2)}{(\lambda+\zeta+1)d_0(\lambda+\zeta+1)}\cot
\zeta\pi,\nonumber \\
=-\frac 1
{2i}\int\limits_{\mathcal{L}_{0,\varepsilon}}\frac{\widetilde{f}(\lambda+\zeta+2)}{\kappa_0(\lambda+\zeta+1)d_0(\lambda+\zeta+1)}[\cot
\zeta\pi+i]d\zeta+\frac{\widetilde{f}(\lambda+2)}{\kappa_0(\lambda+1)d_0(\lambda+1)},
\end{align}
here $\mathrm{Re }\lambda\in (0,\frac{1}{2})$. By \eqref{g5.11}, \eqref{g6.2}, \eqref{g7.2}, \eqref{g7.3}, we have
\begin{align}\label{g11.2}
&\widetilde{q}(\lambda)=\cot \pi \lambda d_1(\lambda)=\cot \pi \lambda d_0(\lambda)y(\lambda)=-\frac{\kappa_0d_0(\lambda+1)}{\lambda}y(\lambda)\nonumber \\ &=-\frac{\widetilde{f}(\lambda+2)}{\lambda(\lambda+1)}
+\frac 1
{2i}\int\limits_{\mathcal{L}_{0,\varepsilon}}\frac{d_0(\lambda+1)}{\lambda(\lambda+\zeta+1)d_0(\lambda+\zeta+1)}\widetilde{f}(\lambda+\zeta+2)[\cot
\zeta\pi+i]d\zeta= \widetilde{q}_1(\lambda)+\widetilde{q}_2(\lambda).
\end{align}

Formula \eqref{g6.3} yields
\begin{equation}\label{g11.3}
    \frac{1}{2i}\frac{d_0(\lambda+1)}{(\lambda+\zeta+1)d_0(\lambda+\zeta+1)}[\cot
\zeta\pi+i]=(\kappa_0\pi)^\zeta\frac{K(\lambda+1)}{(\lambda+\zeta+1)K(\lambda+\zeta+1)}
    \frac{1}{e^{i\pi\zeta}-e^{-i\pi\zeta}}.
\end{equation}
Set
\begin{equation}\label{g12.1}
\Phi(\lambda,\zeta)=(\kappa_0\pi)^\zeta\frac{K(\lambda)}{(\lambda+\zeta)K(\lambda+\zeta)}\frac{1}{e^{i\pi\zeta}-e^{-i\pi\zeta}}.
\end{equation}
Then we get
\begin{equation}\label{g12.3}
    \widetilde{q}_2(\lambda)=\int\limits_{\mathcal{L}_{0,\varepsilon}}\frac{1}{\lambda}\Phi(\lambda+1,\zeta)\widetilde{f}(\lambda+2+\zeta)d\zeta.
\end{equation}
Set $l_0=\mu-k-1$, and then $l_0\in (0,1/2)$. By \eqref{g3.7}
\begin{align}\label{g12.4}
    q(r)=&\frac{1}{2\pi i}\int\limits_{\mathrm{Re }\lambda=l_0}r^{-\lambda}\widetilde{q}(\lambda)d\lambda=\frac{1}{2\pi i}\int\limits_{\mathrm{Re }\lambda=l_0}r^{-\lambda}\widetilde{q}_1(\lambda)d\lambda+\frac{1}{2\pi i}\int\limits_{\mathrm{Re }\lambda=l_0}r^{-\lambda}\widetilde{q}_2(\lambda)d\lambda\nonumber \\
    =&q_1(r)+q_2(r).
\end{align}
Consider $q_{1,r}(r)$ and $q_{2,r}(r)$ separately. For the first term we have
\begin{equation}\label{g12.5}
    q_{1,r}(r)=\frac{1}{2\pi i}\int\limits_{\mathrm{Re}\lambda=l_0}r^{-\lambda-1}\frac{\widetilde{f}(\lambda+2)}{\lambda+1}d\lambda=\frac{1}{2\pi i}\int\limits_{\mathrm{Re}s=l_0+1}r^{-s}\frac{\widetilde{f}(s+1)}{s}ds.
\end{equation}

It is clear that if $u(r)=rf(r)$ then $\widetilde{u}(s)=\widetilde{f}(s+1)$. By formula (7.1.2) in  \cite{BE}, if $\widetilde{v}(s)=s^{-1}$, $\mathrm{Re} s >0$, then
\[
v(r)=\begin{cases}
            1, & 0<x<1,\\
            0, & 1<x.
         \end{cases}
\]
The Mellin convolution formula \eqref{g4.2} gives
\begin{equation}\label{g12.6}
    q_{1,r}(r)=\int\limits^1_{0}\frac{r}{t}f\left(\frac{r}{t}\right)\frac{dt}{t}=\int\limits^\infty_rf(\rho)d\rho.
\end{equation}
By virtue of \eqref{g2.2}, we conclude
\begin{equation}\label{g12.7}
    q_{1,r}(0)=0.
\end{equation}
Next, we consider $q_{2,r}(r)$. Substitution of \eqref{g12.3} to \eqref{g12.4}
gives
\begin{equation}\label{g13.1}
    q_{2,r}(r)=-\frac{1}{2\pi i}\int\limits_{\mathrm{Re} \lambda=l_0}r^{-\lambda-1}d\lambda\int\limits_{\mathcal{L}_{0,\varepsilon}}\Phi(\lambda+1,\zeta)\widetilde{f}(\lambda+2+\zeta)d\zeta.
\end{equation}
The point $\lambda=0$ is not singular point of the last expression (compare with \eqref{g12.3}). Therefore we are able to deduce
\begin{equation}\label{g13.2}
    q_{2,r}(r)=-\frac{1}{2\pi i}\int\limits_{\mathrm{Re} s=l_0+1}r^{-s}ds\int\limits_{\mathcal{L}_{0,\varepsilon}}\Phi(s,\zeta)\widetilde{f}(s
    +1+\zeta)d\zeta=\frac{1}{2\pi i}\int\limits_{\mathrm{Re} s=l_0}r^{-s}ds\int\limits_{\mathcal{L}_{0,\varepsilon}}\Phi(s,\zeta)\widetilde{f}(s
    +1+\zeta)d\zeta.
\end{equation}
Remark that $\Phi(s,\zeta)$ in \eqref{g13.2} is analytic with respect to $s$ if $\mathrm{Re }s\in (l_0,l_0+1)$, $\zeta\in\mathcal{L}_{0,\varepsilon}$ and  $\varepsilon$ is sufficiently small ($\varepsilon<\frac{1}{2}-l_0$). Let us notice that demanding of smallness of $\varepsilon$ was implicitly presented in \cite{SF1}, p.228. Below we will change the contours of integration with respect to $s$, $\zeta$, keeping $\Phi(s,\zeta)$ away from singular points. Throughout we need to be sure that
\begin{equation}\label{g13.3}
0<\mathrm{Re}\, \zeta<1,\quad    0<\mathrm{Re} (s+\zeta)<\frac{3}{2}, \quad -\frac{1}{2}<\mathrm{Re}\, s<\frac{3}{2},
\end{equation}
(cf. \eqref{g6.7}-\eqref{g6.9}).
Set $s=\sigma+i\tau$, $\zeta=\vartheta+i\eta$.

Fix any $\alpha\in (1,2)$ and put
\begin{equation}\label{g13.4}
    \vartheta=l_0+\frac{\alpha-1}{2}.
\end{equation}
This choice will be explained below. Then
\begin{equation}\label{g13.5}
  0<\vartheta<1, 0<\vartheta+l_0=2l_0+\frac{\alpha-1}{2}<\frac{3}{2},
\end{equation}
i.e. \eqref{g13.3} holds for $l_0=\mathrm{Re}\, s$, $\theta=\mathrm{Re}\, \zeta$. Hence we get
\begin{equation}\label{g14.1}
    q_{2,r}(r)=\frac{1}{2\pi i}\int\limits_{\mathrm{Re} s=l_0}r^{-s}ds\int\limits_{\mathrm{Re } \zeta =\vartheta}\Phi(s,\zeta)\widetilde{f}(s
    +1+\zeta)d\zeta
\end{equation}
By the Mellin convolution formula \eqref{g4.2}, we deduce
\begin{align}\label{16.4}
q_{2,r}(r)&=\frac{1}{2 \pi i}\int\limits^\infty_0\frac{dt}{t}\int\limits_{\mathrm{Re }\zeta =\vartheta}d\zeta \int\limits_{\mathrm{Re }s=l_0}f\left(\frac
r t\right)\left(\frac r
t\right)^{1+\zeta}t^{-s}\Phi(s,\zeta)ds\nonumber \\
&=\frac{1}{2 \pi i}\int\limits^\infty_0f\left(\frac
r t\right)\left(\frac r
t\right)^{1+\theta}\frac{dt}{t}\int\limits_{\mathrm{Re }\zeta =\vartheta}d\zeta \int\limits_{\mathrm{Re }s=l_0}\left(\frac r
t\right)^{i\eta}t^{-s}\Phi(s,\zeta)ds\nonumber \\
&=\int\limits^\infty_0f\left(\frac
r t\right)\left(\frac r
t\right)^{1+\vartheta}\mathcal{G}(t,\vartheta)\frac{dt}{t},
\end{align}
here
\begin{equation}\label{g14.4}
    \mathcal{G}(t,\vartheta)=\frac{1}{2\pi i}\int\limits^{+\infty}_{-\infty}\biggl(\frac{r}{t}\biggr)^{i\eta}d\eta\int\limits_{\mathrm{Re }s=l_0}t^{-s}\Phi(s,\vartheta+i\eta)ds
\end{equation}
The Cauchy inequality yields
\begin{equation}\label{g14.5}
    |q'_2(r)|\leq C \left[\int\limits^\infty_0|f\left(\frac
r t\right)|^2\left(\frac r
t\right)^{2(1+\vartheta)}t^{\alpha-2}dt\right]^{\frac{1}{2}}\left[\int\limits^\infty_0|\mathcal{G}(t,\vartheta)|^2t^{-\alpha}dt\right]^{\frac{1}{2}}\leq C I_1^{\frac{1}{2}}I_2^{\frac{1}{2}}.
\end{equation}
Following the choice of $\vartheta$ in \eqref{g13.4}, we obtain
\begin{align}\label{g14.6}
   I_1=\int\limits^\infty_0|f\left(\frac
r t\right)|^2\left(\frac r
t\right)^{2(1+l_0)
+\alpha-1}t^{\alpha-2}dt
=
\int\limits^\infty_0|f\left(
\rho\right)|^2\rho^{2l_0+1+\alpha}\left(\frac{r}{\rho}\right)^{\alpha-2}\frac{rd\rho}{\rho^2}\nonumber \\=r^{\alpha-1}\int\limits^\infty_0|f\left(
\rho\right)|^2\rho^{2(\mu-k-1)+1}d\rho=
r^{\alpha-1}\int\limits^\infty_0|f\left(
\rho\right)|^2\rho^{2(\mu-k)-1}d\rho\leq C r^{\alpha-1}\Vert f\Vert^2_{H^{k+\frac{1}{2}}_\mu(\mathbb{R
}_+)}.
\end{align}
The second integral $I_2$ can be splitted into a sum of two integrals
\[I_2=\int^\infty_0\dots dt=\int^1_0\dots dt+\int^\infty_1\dots dt\]
Let $\sigma_1$,  $\sigma_2$ are chosen to satisfy (compare with \eqref{g13.3})
\begin{align}\label{g15.1}
    \sigma_1>0, 0<-&\sigma_1+\vartheta<\frac{3}{2}, -\frac{1}{2}<-\sigma_1<\frac{3}{2},\nonumber\\
    0<&\sigma_2+\vartheta<\frac{3}{2},  0<\sigma_2<\frac{3}{2}.
\end{align}
Then we derive
\begin{align}\label{g15.2}
    |I_2|\leq C\biggl(\int^1_0t^{2\sigma_1-\alpha}\left|\int\limits^{+\infty}_{-\infty}d\eta\int\limits_{\mathrm{Re }s=-\sigma_1}|\Phi(s,\vartheta+i\eta)|ds\right|^2dt\nonumber \\
    +\int^\infty_1t^{-2\sigma_2-\alpha}\left|\int\limits^{+\infty}_{-\infty}d\eta\int\limits_{\mathrm{Re }s=\sigma_2}|\Phi(s,\vartheta+i\eta)|ds\right|^2dt\biggr).
\end{align}
We claim that
\begin{equation}\label{g15.3}
|\Phi(s,\zeta)|\leq C(M,\sigma,\vartheta)\exp\left(-\frac \pi
4|\eta|\right)\left[\frac 1
{(1+|\tau|)^{1+\vartheta}}+\exp\left(-\frac \pi 4
|\tau|\right)\right],
\end{equation}
under assumptions (cf. \eqref{g13.3})
\begin{equation}\label{g15.4}
    \vartheta\in (0,1), 0<\mathrm{Re \,}s+\vartheta <\frac{3}{2}, -\frac{1}{2}<\mathrm{Re }s<\frac{3}{2}.
    \end{equation}
The proof \eqref{g15.3} is given in Appendix B. Estimate \eqref{g15.3} yields
\begin{equation}\label{g15.5}
    \int\limits^{+\infty}_{-\infty}d\eta\int\limits_{\mathrm{Re }s=\sigma_0}|\Phi(s,\vartheta+i\eta)|ds\leq C(\sigma_0,\vartheta),
\end{equation}
for any $\sigma_0=\mathrm{Re} s$ and $\vartheta$ satisfying \eqref{g15.4}.

In addition to \eqref{g15.1} we need
\begin{equation}\label{g15.6}
    2\sigma_1-\alpha>-1, -2\sigma_2-\alpha<-1.
\end{equation}
It can be easily checked that \eqref{g15.1}, \eqref{g15.6},
\eqref{g15.5} guarantee the convergence of integrals in \eqref{g15.2}.

Now we proceed to the choice of $\sigma_1$, $\sigma_2$. Let us remind that $\theta$ are already chosen.  It is easily seen that any
\begin{equation}\label{g16.1}
    \sigma_2\in (0,\frac{3}{2}-\vartheta)
\end{equation}
satisfies both \eqref{g15.1} and \eqref{g15.6}. Parameter $\sigma_1$ satisfies \eqref{g15.1},\eqref{g15.6}, if
\begin{equation}\label{g16.2}
    \frac{\alpha-1}{2}<\sigma_1, \sigma_1<\vartheta, \sigma_1<\frac{1}{2}.
\end{equation}
We check at once that
\begin{equation*}
  \frac{\alpha-1}{2}<\frac{1}{2}, \quad \frac{\alpha-1}{2}<\vartheta=l_0+\frac{\alpha-1}{2},
\end{equation*}
i.e. interval $\left(\frac{\alpha-1}{2},\min\{\frac{1}{2},\vartheta\}\right)$ is not empty. Hence we can  take any
\begin{equation}\label{g16.4}
    \sigma_1\in \left(\frac{\alpha-1}{2},\min\left\{\frac{1}{2},\vartheta\right\}\right).
\end{equation}
This allows us  finally to conclude
\begin{equation}\label{g16.5}
    I_2\leq C(\sigma_1,\sigma_2,\theta).
\end{equation}
Substitution \eqref{g14.6}, \eqref{g16.5} in \eqref{g14.5} gives
\begin{equation}\label{g16.6}
    |q_{2,r}(r)|\leq C r^{\frac{\alpha-1} {2}}\Vert f\Vert^2_{H^{k+\frac{1}{2}}_\mu(\mathbb{R
}_+)}.
\end{equation}
Combining \eqref{g12.7} and \eqref{g16.6} we can assert that
$q'(0)=0$, that justifies \eqref{g1.5}. The proof of Theorem 1 is finished.

\begin{remark}\label{r17.1}
Let $\widehat{f}$ be a polynomial of degree $m$
\[\widehat{f}(x_1)=\sum\limits^m_{k=0} f_kx_1^k.\]
One can easily checks that
\begin{align*}
\widehat{q}(x_1)=&q_0+q_1x_1-\sum\limits^m_{k=0}
\frac{f_k}{(k+1)(k+2)}x_1^{k+2},\\
\widehat{p}(x)=&q_0+q_1x_1-\sum\limits^m_{k=0}
\frac{f_k}{(k+1)(k+2)}\mathrm{Re \, }\left[\left(x_1+ix_2\right)^{k+2}\right],
\end{align*}
satisfy \eqref{g1.1}-\eqref{g1.4}. In this context, it may be the
subject of a future paper to consider problem
\eqref{g1.1}-\eqref{g1.6} without artificial condition
$\int\limits^0_{-\infty}\widehat{f}(x_1)dx_1= 0,$ in
weighted Sobolev spaces with nonhomogeneous norms (cf. \cite{CDN}, \cite{SF2}, \cite{NZ}).
\end{remark}

\section{Appendix A}\label{s5}

We have immediately from \eqref{g8.2}
\begin{equation}\label{fa.1}
    p(r,\theta)=\frac{1}{2\pi i}\int\limits_{\mathrm{Re }\lambda=l_0}r^{-\lambda}\frac{\cos \lambda\theta}{\sin \lambda\pi}d_1(\lambda)d\lambda,
    \quad q(r)=\frac{1}{2\pi i}\int\limits_{\mathrm{Re }\lambda=l_0}r^{-\lambda}\cot s\pi d_1(\lambda)ds,
\end{equation}
here $l_0=\mu-k-1\in (0,1/2)$. Then it follows
\begin{equation}\label{ga.2}
    p(r,\pi)+p(r,-\pi)-2q(r)=
    \frac{1}{2\pi i}\int\limits_{\mathrm{Re }\lambda=l_0}r^{-\lambda}[\cot \lambda\pi+\cot \lambda\pi-2\cot \lambda\pi] d_1(\lambda)d\lambda=0.
\end{equation}
This gives \eqref{g1.2}.

Direct computations shows that
\begin{align}\label{bg1}
    p_\theta(r,\theta)&=-\frac{1}{2\pi i}\int\limits_{\mathrm{Re }\lambda=l_0}r^{-\lambda}\frac{\sin \lambda\theta}{\sin \lambda\pi}\lambda d_1(\lambda)d\lambda,\nonumber
    \\
    p_\theta(r,\pm\pi)&=\mp\frac{1}{2\pi i}\int\limits_{\mathrm{Re }\lambda=l_0}r^{-\lambda}\lambda d_1(\lambda)d\lambda,
    \end{align}
and
\begin{equation*}
    p(r,\pi)-p(r,-\pi)=0.
\end{equation*}
This gives \eqref{g5.3}, and hence \eqref{g1.3}. Let us notice that, as easily seen, $p_\theta(r,\pi)-p_\theta(r,-\pi)\neq 0$. Thus, equation \eqref{g5.4} is not reduced to $-q_{rr}(r)=f(r)$.

Next, we proceed to condition \eqref{g5.4}. We have
\begin{align*}
    B[p,q](r)\overset{def}{=} &-q_{rr}(r)+\frac{\kappa_2}{r}\left(p_\theta(r,\pi)-p_\theta(r,-\pi)\right)\\ =&
    -\frac{1}{2\pi i}\int\limits_{\mathrm{Re }\lambda=l_0}r^{-\lambda-2}\cot \lambda\pi \lambda(\lambda+1)d_1(\lambda)d\lambda-\frac{2\kappa_2}{2\pi i}\int\limits_{\mathrm{Re }\lambda=l_0}r^{-\lambda-1}\lambda d_1(\lambda)d\lambda.
\end{align*}
Since $d_1(\lambda)$ solves \eqref{g6.1}, we can continue
\begin{align*}
    B[p,q](r)&=\frac{1}{2\pi i}\int\limits_{\mathrm{Re }\lambda=l_0}r^{-\lambda-2}\widetilde{f}(\lambda+2)d\lambda\\
		&+\frac{\kappa_0}{2\pi i}\int\limits_{\mathrm{Re }\lambda=l_0}\left[r^{-\lambda-2}(\lambda+1)d_1(\lambda+1)d\lambda-r^{-\lambda-1}\lambda d_1(\lambda)\right]d\lambda\\
    &=\frac{1}{2\pi i}\int\limits_{\mathrm{Re }\lambda=l_0+2}r^{-\lambda}\widetilde{f}(\lambda)d\lambda\\
		&+\frac{\kappa_0}{2\pi i}\left[\int\limits_{\mathrm{Re }\lambda=l_0+1}r^{-\lambda-1}\lambda d_1(\lambda)d\lambda- \int\limits_{\mathrm{Re }\lambda=l_0}r^{-\lambda-1}\lambda d_1(\lambda)d\lambda \right]=f(r),
\end{align*}
here the expression in the square brackets is equal to zero since $d_1(\lambda)$ is analytic in the strip \[\mathrm{Re } \lambda \in (l_0,l_0+1)\] (cf. \eqref{g7.4}).
This gives \eqref{g5.4} and hence \eqref{g1.4}.

\section{Appendix B}\label{s6}

Preliminarily we consider in plane $(\eta,\tau)$ the sets
\begin{align}\label{9.12}
Q_1&=\{ \eta\leq 0, - 2 \eta<\tau \}\cup \{ \eta>0,
\tau>-\frac 2 3\eta\},\nonumber\\
Q_2&=\{ \eta\leq 0, \tau<-\frac 2 3 \eta\}\cup \{ \eta>0,
\tau<- 2 \eta\},\nonumber\\
Q_3&=\{ \tau>0, -\frac 2 3\eta<\tau<-2\eta\}, \nonumber \\
Q_4&=\{ \tau<0, -2 \eta<\tau<-\frac 2 3\eta\}.
\end{align}
It follows immediately that
\begin{align}\label{9.13}
|\tau|\leq 2|\eta+\tau|, (\eta,\tau)\in Q_1\cup Q_2,\nonumber \\
|\tau|\geq 2|\eta+\tau|, (\eta,\tau)\in
Q_3\cup Q_4.
\end{align}

To deal with estimates of $\Phi(s,\zeta)$, we will partition the
plane $(\eta,\tau)$ into the following sets
\begin{align*}
&B_0=\{(\eta,\tau):\, |\eta|\leq 3M, |\tau|\leq 2M\}, \\
&B_1=\{(\eta,\tau):\, \eta> 3M, |\tau|\leq 2M\},  \\
&B_{1.1}=\{(\eta,\tau):\, -\frac 3 2\tau<\eta, \tau<-2M\},\\
&B_{1.2}=\{(\eta,\tau):\, \eta>0, \tau> 2M\},  \\
&B_{1.3}=\{(\eta,\tau):\, -\frac \tau 2<\eta<0, \tau> 2M\},\\
&B_{3.1}=\{(\eta,\tau):\, M-\tau<-\frac \tau 2, \tau> 2M\},\nonumber  \\
&B_{3.2}=\{(\eta,\tau):\, -M-\tau<M-\tau, \tau> 2M\},\nonumber  \\
&B_{3.3}=\{(\eta,\tau):\, -\frac 3 2\tau<\eta<-M-\tau, \tau> 2M\},\nonumber  \\
&B_{2.1}=\{(\eta,\tau):\, \eta<-\frac 32 \tau, \tau> 2M\},\nonumber  \\
&B_{2}=\{(\eta,\tau):\, \eta<-3M, |\tau|< 2M\},\nonumber \\
&B_{2.2}=\{(\eta,\tau):\, \eta<0, \tau<-2M\},\nonumber \\
&B_{2.3}=\{(\eta,\tau):\, 0<\eta<-\frac \tau 2, \tau<- 2M\},\nonumber  \\
&B_{4.1}=\{(\eta,\tau):\, -\frac \tau 2 <\eta<-\tau -M, \tau<-2M\},\nonumber  \\
&B_{4.2}=\{(\eta,\tau):\, -\tau-M<\eta<M-\tau, \tau<-2M\},\nonumber  \\
&B_{4.3}=\{(\eta,\tau):\, M-\tau<\eta< -\frac 3 2 \tau, \tau<-2M\}.\nonumber
\end{align*}
Here we choose $M>0$ sufficiently large to neglect $r_1(\lambda)$ in
\eqref{g6.10} and use \eqref{g6.17}. This construction was motivated
by \cite{BV1}, \cite{DV}.

By assumptions \eqref{g15.4}, we clearly have
\begin{equation}\label{4.6b}
    |\Phi(s,\zeta)|\leq C(M, \vartheta), \mbox{ in } B_0.
\end{equation}
Now estimate \eqref{g15.3} will follow from the four lemmas below.

\begin{lemma}\label{l11.1}
Under conditions \eqref{g15.4}, we have
\begin{equation}\label{11.4}
|\Phi(s,\zeta)|\leq C(M,\vartheta)\exp(-\pi |\eta|), \mbox{ in }
B_1, B_2.
\end{equation}
\end{lemma}
\begin{proof} It is clear that
\begin{equation}\label{11.5}
|K(s)|\leq C(M), \quad\left|\frac 1 {e^{i\zeta \pi}-e^{-i\zeta \pi}}\right|\leq
C(M,\vartheta)\exp(-\pi|\eta|), \mbox{ with fixed } \vartheta\in (0,1),
\end{equation}
in $B_1$, $B_2$. Now \eqref{g6.10}, \eqref{11.5} give
\begin{equation}\label{11.6}
|\Phi(s,\zeta)|\leq
C(M,\vartheta)|s+\zeta|^{-1/2-\sigma-\vartheta}\exp((\tau+\eta)\arg\omega(s+\zeta)-\pi|\eta|),\mbox{
in } B_1, B_2.
\end{equation}
Remark that $|\tau+\eta|\to \infty$, when
$|\eta|\to \infty$, and, by \eqref{g15.4}, \eqref{g6.17}, we get
\begin{align}\label{11.7}
\exp((\tau+\eta)\arg\omega(s+\zeta))\leq C(M), \nonumber \\
|s+\zeta|^{-1/2-\sigma-\vartheta}\leq C(M).
\end{align}
From the above it follows \eqref{11.4}.
\end{proof}

Since $|\tau|$ is bounded in $B_1$, $B_2$, \eqref{11.4} implies \eqref{g15.3} in these sets.

\begin{lemma}\label{l11.2}
Under conditions \eqref{g15.4}, we have
\begin{equation}\label{11.8}
|\Phi(s,\zeta)|\leq
C(M,\sigma,\vartheta)\tau^{-1-\vartheta}\exp(-(\pi-\varepsilon)
|\eta|), \mbox{ in } B_{1.1}, B_{1.2}, B_{1.3},
B_{2.1}, B_{2.2}, B_{2.3},
\end{equation}
here $\varepsilon$ is sufficiently small and depends on $M$.
\end{lemma}
\begin{proof}
 Our proof starts with observation that, by definition,
\begin{equation}\label{11.9}
|\tau+\eta|\geq M,
\end{equation}
and, by \eqref{9.13},
\begin{equation}\label{11.10}
\frac 1 {|\tau+\eta|}\leq \frac 2 {|\tau|}, \mbox{ in }
 B_{1.1}, B_{1.2}, B_{1.3},
B_{2.1}, B_{2.2}, B_{2.3}.
\end{equation}
We will consider separately the set $B_{1.2}$, where
$\mathrm{sgn \,} \tau=\mathrm{sgn \,}\eta$, and the set
$B_{1.3}$, where $\mathrm{sgn \,} \tau=-\mathrm{sgn
\,}\eta$.

From \eqref{g6.10} we deduce that
\begin{align}\label{12.1}
|\Phi(s,\zeta)|\leq
C(M,\vartheta)\exp(-\pi|\eta|)\left|\frac{\omega(s)}{\omega(s+\zeta)}\right|^{\sigma-1/2}\frac
1{|\omega(s+\eta|^{1+\vartheta}}\nonumber \\ \times\exp(-\tau
\arg\omega(s)+(\tau+\eta)\arg \omega(s+\zeta)).
\end{align}
We have
\begin{equation}\label{12.2}
\exp(-\tau \arg\omega(s)+(\tau+\eta)\arg
\omega(s+\zeta))=\exp(\eta\arg \omega(s+\zeta))\exp(\tau(
\arg\omega(s+\zeta)-\arg \omega(s))).
\end{equation}

We claim that $|\eta\arg \omega(s+\zeta)|$ and $|\tau(
\arg\omega(s+\zeta)-\arg \omega(s))|$ are uniformly bounded in
$B_{1.2}$, $B_{1.3}$.

For the first expression, estimate \eqref{g6.17} implies
\begin{equation}\label{12.3}
|\eta\arg \omega(s+\zeta)|\leq C(M)\frac{|\eta|}{|\eta+\tau|} \mbox{ in } B_{1.2}, B_{1.3}.
\end{equation}
In $B_{1.2}$ we get
\begin{equation}\label{12.4}
\frac{|\eta|}{|\eta+\tau|}\leq 1,
\end{equation}
since $\mathrm{sgn \,} \tau=\mathrm{sgn \,}\eta$. In
$B_{1.3}$, we have $|\eta|\leq \frac {|\tau|}2$ and
use \eqref{11.10}
\begin{equation}\label{12.5}
\frac{|\eta|}{|\eta+\tau|}\leq \frac{|\tau|}{2|\eta+\tau|}\leq
\frac{|\tau|}{|\tau|}=1.
\end{equation}

In order to estimate $|\tau( \arg\omega(s+\zeta)-\arg \omega(s))|$
we use the formula
\begin{equation}\label{12.8}
\arctan a -\arctan b= \arctan \frac{a-b}{1+ab}.
\end{equation}
and \eqref{g6.17}. Then
\begin{align}\label{12.7}
\arg \omega(s)=\arctan\frac{\Psi_1(s)}{\Psi_2(s)},\quad \arg
\omega(s+\zeta)=\arctan\frac{\Psi_1(s+\zeta)}{\Psi_2(s+\zeta)},\nonumber
\\ \Psi_1(s)=\tau\sin 2\pi\sigma-\sigma\sinh 2\pi \tau,\quad \Psi_2(s)=\sigma\sin 2\pi\sigma+\tau\sinh 2\pi
\tau.
\end{align}
Along this way we deduce in $B_{1.2}$,
$B_{1.3}$
\begin{align}\label{12.9}
\arg \omega(s+\zeta)-\arg \omega(s)=\arctan \Psi_0(s,s+\zeta),\nonumber \\
\Psi_0(s,s+\zeta)=\frac{\Psi_1(s)\Psi_2(s+\zeta)-\Psi_1(s+\zeta)\Psi_2(s)}{\Psi_2(s)\Psi_1(s+\zeta)+\Psi_1(s)\Psi_2(s+\zeta)}.
\end{align}
A detailed analysis of the function $\Psi_0(s,s+\zeta)$ shows that
the following estimate holds
\begin{align}\label{13.2}
|\Psi_0(s,s+\zeta)|\leq C_1(M,\sigma,\vartheta)\frac
{|(\sigma+\vartheta)\tau-\sigma(\tau+\eta|}{|\tau||\tau+\eta|}+\nonumber
\frac{C_2(M,\sigma,\vartheta)}{|\tau|}
\\
\leq C_1(M,\sigma,\vartheta)\left[\frac 1 {|\tau+\eta|}+\frac
{|\eta|}{|\tau||\tau+\eta|}\right]+\frac{C_2(M,\sigma,\vartheta)}{|\tau|}.
\end{align}
In $\Sigma^{\prime}_2$ we have $\tau>2M$, $\eta>0$ and therefore
\begin{equation}\label{13.3} |\Psi_0(s,s+\zeta)|\leq
 C_1(M,\sigma,\vartheta)\left[\frac 1 {|\tau|}+\frac
{|\eta|}{|\tau||\eta|}\right]+\frac{C_2(M,\sigma,\vartheta)}{|\tau|}\leq
\frac{C(M,\sigma,\tau)}{|\tau|}.
\end{equation}
By \eqref{9.13}, we have $\frac 1 {|\eta+\tau|}\leq \frac 2
{|\tau|}$ in $\Sigma^{\prime\prime}_2$ and, besides,
$-|\eta|>-\frac{|\tau|}{2}$. Thus we get
\begin{equation}\label{13.4}
|\Psi_0(s,s+\zeta)| \leq C_1(M,\sigma,\vartheta) \frac 2 {|\tau|}
\left[1+\frac
{|\eta|}{|\tau|}\right]+\frac{C_2(M,\sigma,\vartheta)}{|\tau|}\leq
\frac{C(M,\sigma,\tau)}{|\tau|}.
\end{equation}
Combining of \eqref{12.9}, \eqref{13.2}, \eqref{13.3} yields
\begin{equation}\label{13.5}
|\tau( \arg\omega(s+\zeta)-\arg \omega(s))|\leq C(M,\sigma,\vartheta),
\mbox{ in } B_{1.2}, B_{1.3}.
\end{equation}
From \eqref{12.2},\eqref{12.4}, \eqref{12.5}, \eqref{13.5} we obtain
\begin{equation}\label{13.6}
|\exp(-\tau \arg\omega(s)+(\tau+\eta)\arg \omega(s+\zeta))|\leq
C(M,\sigma,\vartheta)
\end{equation}
Taking into account \eqref{12.1}, \eqref{13.6}, we conclude that
\eqref{11.8} is valid in $B_{1.2}$, $B_{1.3}$. In a similar fashion one can consider the
sets $B_{1.1}$,   $B_{2.1}$, $B_{2.2}$,
$B_{2.3}$.
\end{proof}

\begin{lemma}\label{l13.1}
We have
\begin{equation}\label{13.7}
|\Phi(s,\eta)|\leq C(M,\sigma,\vartheta)\exp(-(\frac \pi
2-\varepsilon)|\eta|)\exp(-\frac \pi 4|\tau|), \mbox{ in } B_{3.1},
B_{3.3}, B_{4.1},
B_{4.3},
\end{equation}
under assumptions \eqref{g15.4}.
\end{lemma}
\begin{proof} In $B_{3.1},
B_{3.3}, B_{4.1},
B_{4.3},$ the inequality $|\tau|\leq
2|\eta+\tau|$ no longer works as in the previous lemmas but we can
still suggest that $|\tau|$, $\eta$, $|\tau+\eta|$ are large and use
asymptotics \eqref{g6.10}
\begin{align}\label{13.8}
|\Phi(s,\zeta)|\leq
C(M,\sigma,\vartheta)\exp(-\pi|\eta|)\frac{|\eta|^{\sigma-1/2}}{|\omega(s+\zeta)|^{1+\vartheta}}\exp(-\tau
\arg\omega(s)+(\tau+\eta)\arg \omega(s+\zeta)).
\end{align}
With the help of calculations as above we infer
\begin{align}\label{14.1}
|(\tau+\eta)\mathrm{\, \arg \,} \omega(s+\zeta)-\mathrm{\, \arg \,} \omega(s)|=|(\tau+\eta)(\arg \omega(s+\zeta)-\tau \mathrm{\, \arg \,}
\omega(s))+\eta \mathrm{\, \arg \,} \omega(s)|\nonumber \\ \leq
C(M,\sigma,\vartheta)|\tau+\eta|\left[
\frac{|\tau|}{|\tau||\tau+\eta|}
+\frac{|\eta|}{|\tau||\tau+\eta|}\right]\leq C(M,\sigma,\vartheta)\left(1+\frac{|\eta|}{|\tau|}\right).
\end{align}
By definitions of considered sets, we have
\begin{equation}\label{14.2}
\frac{|\eta|}{|\tau|}\leq \frac 3 2, \exp(-|\eta|)\leq
C(M)\exp(-\frac {|\tau|}2).
\end{equation}
Applying \eqref{13.8}, \eqref{14.1}, \eqref{14.2} gives
\begin{align}\label{14.3}
|\Phi(s,\zeta)|&\leq
C(M,\sigma,\vartheta)|\eta|^{\sigma-1/2}\exp(-\frac \pi
2|\eta|)\exp(-\frac \pi 4|\tau|)\nonumber \\
&\leq C(M,\sigma,\vartheta)|\exp(-(\frac \pi
2-\varepsilon)|\eta|)\exp(-\frac \pi 4|\tau|).
\end{align}
This completes the proof of \eqref{13.7}.
\end{proof}

\begin{lemma}\label{l14.1} Let $\sigma$, $\vartheta$ satisfy
\eqref{g15.4}. Then
\begin{equation}\label{14.4}
|\Phi(s,\zeta)|\leq C(M,\sigma,\vartheta) \exp(-(\frac \pi
2-\varepsilon)|\tau|)\exp(-\frac \pi 2|\eta|), \mbox{ in } B_{3.2},
B_{4.2}.
\end{equation}
\end{lemma}
\begin{proof} We observe that $|\tau+\eta|\leq M$ in $B_{3.2}$,
$B_{4.2}$, or, in more details
\begin{align}\label{14.5}
-M-|\tau|<-|\eta|<M-|\tau|, \mbox{ in } B_{3.2},\nonumber \\
|\tau|-M<|\eta|<M+|\tau|, \mbox{ in } B_{4.2},
\end{align}
and, consequently,
\begin{equation}\label{14.6}
\exp(-|\eta|)\leq C(M) \exp(-|\tau|).
\end{equation}
By \eqref{g6.10}, \eqref{g6.14}, it can be easily verified that
\begin{align}\label{14.7}
|\Phi(s,\zeta)|\leq C(M,\sigma,\vartheta)\exp(-\pi
|\eta|)|\omega(s)|^{\sigma-1/2}\exp(-\tau \arg \omega(s))\nonumber \\
\leq C(M,\sigma,\vartheta)\exp(-\frac \pi 2 |\eta|)\exp(-(\frac \pi
2 -\varepsilon)|\tau|),
\end{align}
in $B_{3.2}$, $B_{4.2}$.
\end{proof}

Summing up estimates \eqref{4.6b}, \eqref{11.4}, \eqref{11.8}, \eqref{13.7},
\eqref{14.4}, we finish the proof of \eqref{g15.3}.

\subsection*{Acknowledgments}
This work was partially supported by a grant from the Simons Foundation (SFI-PD-Ukraine-00017674, Mykola Krasnoshchok).

\end{document}